\documentclass{barbara}

\usepackage{advdate}

\title{Symplectic Branching through Crystals}
\author{Bárbara Muniz}
\date{\AdvanceDate[-1]{\today}}

\usepackage{indentfirst}
\usepackage{lipsum}

\ytableausetup{centertableaux}

\newcommand{\qbinom}[3][q]{\left[\begin{matrix} #2 \\ #3 \end{matrix}\right]_{#1}}
\newcommand{\e}{\widetilde{e}_{i}}
\newcommand{\f}{\widetilde{f}_{i}}
\begin{document}

\maketitle

\begin{abstract}
We give an alternative proof of Naito--Sagaki's conjecture, which states that the restriction of $gl_{2n}(\CC)$-representations to $sp_{2n}(\CC)$ can be described in terms of crystals.
Using the tableau model for crystals, we construct an explicit and self-contained bijection between their highest weight elements and Sundaram's branching model.
\end{abstract}

\vspace{5pt}
\section*{Introduction}

A classic problem in representation theory, known as \emph{symplectic branching}, asks how $gl_{2n}(\CC)$-representations decompose when viewed as the restricted $sp_{2n}(\CC)$-representations. That is, given finite dimensional (irreducible) $gl_{2n}(\CC)$ and $sp_{2n}(\CC)$ representations $V$ and $W$, we would like to explicitly determine 
\[[V:W]:=dimHom_{sp_{2n}}(V,W)\]
which should be regarded as the multiplicity of $W$ in $V$.\\

This question was first addressed by Littlewood \cite{littlewood},  who proved a formula for sufficiently large $n$. Later, Sundaram \cite{Sundaram} adapted this rule for the general case, fully solving the problem.

\begin{thr*}[Sundaram]
\label{sundaram}
Let $\lambda$ and $\mu$ be two partitions with at most $2n$ and $n$ parts, respectively.
The multiplicity of the irreducible component $V_{sp_{2n}}(\mu)$ in $V_{gl_{2n}}(\lambda)$, as the restricted $sp_{2n}(\CC)$-representation, is given by
    \[[V_{gl_{2n}}(\lambda) : V_{sp_{2n}}(\mu)]=\sum_{\delta}(sp_nc)^{\lambda}_{\mu, (2\delta)'}\]
where $(sp_nc)^{\lambda}_{\mu, \nu}$ is the number of $n$-symplectic Littlewood--Richardson tableaux of shape $\lambda\setminus \mu$ and weight $\nu$.
\end{thr*}

Another branching rule, due to Kwon \cite{kwon}, provides an answer for a wider range of branchings. These descriptions were then applied by Lecouvey--Lenart \cite{lecouvey_lenard} to provide a combinatorial formula of generalized exponents in type C. A connection between Kwon's and Sundaram's rules was shown by Kumar--Tores \cite{kumar-torres}, who used hive models \cite{knutson_tao} to construct a bijection between the tableaux counted by them.\newpage

Our aim in this work is to show that these multiplicities can be read from the crystals when using their tableaux model.\\

\begin{restatable}{main_thr}{branching}\label{branching}
Let $\lambda$ and $\mu$ be two partitions with at most $2n$ and $n$ parts, respectively.\\
The multiplicity of the irreducible component $V_{sp_{2n}}(\mu)$ in $V_{gl_{2n}}(\lambda)$ as the restricted $sp_{2n}(\CC)$-representation is the number of $sp_{2n}$-highest weight semistandard Young tableaux $Y$ of shape $\lambda$ and weight $\mu$.\\
\end{restatable}

This result was first conjectured by Naito--Sagaki \cite{naito-sagaki}, in terms of the equivalent Littelmann's path model \cite{littelman}. They also confirmed it in the special cases where the weight $\lambda$ corresponds to a hook or to a rectangular diagram. The conjecture was then proven in full generality by Schumann--Torres \cite{schumann_torres}, who constructed a bijection between Sundaram's symplectic tableaux and the $U_q(sp_{2n})$-highest weight vectors in the crystal. Their construction made use of some heavy machinery; specifically, it builds upon Burge's correspondence \cite{burge} and a symplectic version of the Robinson-Schensted-Knuth correspondence constructed by Sundaram \cite{sundaram_thesis}.\\

Recently, a new proof to the conjecture was presented by Naito--Suzuki--Watanabe \cite{naito-suzuki-watanabe}, using Watanabe's alternative branching rule \cite{watanabe}. 
In these works, the question is considered by using a different deformation of $sp_{2n}(\CC)$, built as a coideal subalgebra of $U_q(gl_{2n})$  to  have a natural embedding \cite{Kolb}. \\

Our approach is similar to the one presented by Schumann--Torres, in the sense that we also prove the statement by means of a bijection between the same sets. However, our bijection is elementary and self-contained --- it can be proven independently, without relying on intricate theory or external correspondences.\\

We show that the $sp_{2n}$-highest weight tableaux can be split into a fixed part of shape $\mu$ and a semistandard part of shape $\lambda \setminus \mu$ filled by ordered $\{i, \overline{\,i\,}\}$-pairs, as illustrated below
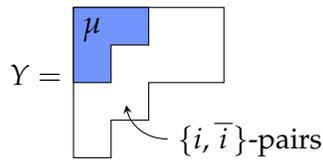
\[\begin{tikzpicture}[scale=0.5]
    \vspace{-10pt}
    \draw (0,0) -- (0,4) -- (4,4) -- (4,2) -- (2,2) -- (2,1) -- (1,1) -- (1,0) -- cycle;
    \filldraw[fill=blue!70] (0,2) -- (0,4) -- (2,4) -- (2,3) -- (1,3) -- (1,2) -- cycle;
    \node[anchor=center] at (-1,2.2){$Y=$};
    \node[anchor=center] at (0.5,3.4) {\large{$\mu$}};
    \node[anchor=west] at (2.5,0.5) {$\{i, \overline{\,i\,}\}$-pairs};
    \draw[-stealth] (2.5,0.5) to[bend left=40] (1.4,1.4);
\end{tikzpicture}
    \vspace{-8pt}
\]
and each of these ordered pairs can be seen as a copy of $\epsilon_1+\epsilon_2=\tinydiagram{1,1}$\,, resulting in a weight of the desired form $(2\delta)'$.\\ 

Intuitively, we see our bijection as a map between skew-tableaux of the same shape, working as a translator of type C conditions into their type A counterparts. In practice, we construct two simple maps $\iota_{sp}$ and $\iota_{LR}$ which delete one by one the $\{i, \overline{\,i\,}\}$-pairs and their corresponding fillings in $(2\delta)'$. This correspondence allows us to then define our bijection inductively in the size of the skew-tableaux as $F=\iota_{sp}^{-1}\circ F\circ \iota_{LR}$.\newpage

We first tackle this problem in the stable case, where $n>\ell(\lambda)$ and, therefore, every Littlewood--Richardson tableau is $n$-symplectic. Then, through this approach, what differs when $n$ is not big enough becomes clear: some of Littlewood--Richardson tableaux are mapped to tableaux with fillings bigger than $n$, which are not in $\Bscr_{gl_{2n}}(\lambda)$. The tableaux that cause this problem can be identified and we can extend the result to the general case.

\subsection*{Acknowledgements}
I would like to thank Prof. Catharina Stroppel for suggesting me this problem and advising me on the master's thesis that would culminate in this paper. I am also grateful to Prof. Jacinta Torres for her support in getting this paper to the finish line.

This research was partially supported by 
{\it Narodowe Centrum Nauki}, grants 2021/43/D/\\ST1/
02290 and 2021/42/E/ST1/00162.
For the purpose of Open Access, the author has applied a CC-BY public copyright license to any Author Accepted Manuscript (AAM) version arising from this submission.

\section{Preliminaries}\subsection{The general and symplectic Lie algebras}
Consider the \emph{general Lie algebra} $gl_{n}(\CC)$ generated over $\CC$ by
\[e_i=E_{i, i+1} \hspace{1cm} f_i=E_{i+1, i}
\hspace{1cm} \hfrak=span_{\CC}\{E_{i,i}\}_{i\leq n}
\]
and on the \emph{symplectic Lie algebra} $sp_{2n}(\CC)$ generated  over $\CC$ by
\begin{align*}
&e_i= E_{i, i+1}-E_{i, i+1}^{aT} &\hspace{1cm}& e_n=E_{n, n+1}\\
&f_i=E_{i+1, i}-E_{i+1, i}^{aT} &\hspace{1cm}& f_n=E_{n+1, n}\\
&h_i=E_{i,i}-E_{i+1, i+1}-E_{i,i}^{aT}+E_{i+1,i+1}^{aT} &\hspace{1cm}& h_n=E_{n,n}-E_{n+1, n+1}
\end{align*}
where $aT$ is the anti-transpose and $E_{i,j}$ is the matrix with a single non-zero entry 1 at $(i,j)$.\\
These are also referred to as the Lie algebras of type $A_{n-1}$ and $C_{n}$, respectively.\\

Their dominant integral weights, which correspond to the irreducible representations are
\begin{align*}
    &X^+_{gl}=\{\lambda=\lambda_1\epsilon_1+\cdots+\lambda_n\epsilon_n\mid \lambda_1\ge \cdots \ge \lambda_n \in \ZZ\}\\
    &X^+_{sp}=\{\lambda=\lambda_1\epsilon_1+\cdots+\lambda_n\epsilon_n\mid \lambda_1\ge \cdots \ge \lambda_n \ge 0 \in \ZZ\}
\end{align*}
where $\epsilon_i$ is the weight sending a diagonal matrix to its $i-th$ entry.

\begin{rmk}
    In type $A_{n-1}$, we restrict our attention to $X^+_{\ge 0} = X^+ \cap \text{span}_{\ZZ}\{\epsilon_i\}_{i\leq n}$ because every irreducible representation corresponds to an element in $X^+_{\ge 0}$ via $V(\lambda)=V(\lambda_{\ge 0})\otimes tr^{\otimes \lambda_n}$ where $\lambda_{\ge 0}=\lambda - \lambda_n (1, 1, \cdots, 1)$.
\end{rmk}

\subsection{Young tableaux}

Let $\lambda=(\lambda_1 \geq \cdots \geq \lambda_n)$ and $\mu=(\mu_1\geq \cdots \geq \mu_m)$ be partitions and let $\Acal$ be a countable totally ordered set. A \emph{Young tableau} $Y$ of \emph{shape} $\lambda$ in the \emph{alphabet} $\Acal$ is an assignment of an element of $\Acal$ to each entry $(i, j)$ with $j\leq \lambda_i$ and \emph{skew tableau} of shape $\lambda\setminus\mu$ in the same alphabet is such an assignment to each entry $(i, j)$ with $\mu_i \leq j\leq \lambda_i$. We call the images $a=Y(i,j)$ its \emph{fillings} and define its \emph{weight} to be the partition $\nu=(\nu_1 \geq \cdots \geq \nu_n)$ where $\nu_k=\#\{(i,j)\mid Y(i, j)=a_k\}$.\\

We say a tableau is \emph{semistandard} if the entries weakly increase along rows and strictly increase along columns, i.e. $Y(i, j)<Y(i+1, j)$ and $Y(i,j)\leq Y(i, j+1)$ for all $i,j$.\\

Furthermore, we define \emph{(far-easter) reading} of a tableau $Y$ of shape $\lambda$ to be a vector $v_Y\in span_\CC\Acal^{\otimes N}$ for $N=|\lambda|$ obtained by recording the fillings of each column top to bottom and then concatenating them right to left. 
\begin{ex} 
\[\ytableausetup{smalltableaux}
Y=\ytableaushort{123,45,6}\;\;\longmapsto \;\; v_{Y}=\boxed{3}\otimes\boxed{2}\otimes\boxed{5}\otimes\boxed{1}\otimes\boxed{4}\otimes\boxed{6}\]\vspace{2pt}
\end{ex}

A \emph{Littlewood--Richardson tableau} is a semistandard skew-tableau such that every prefix of its (far-eastern) reading corresponds to a partition, i.e. has at least as many $\boxed{i}$ s as $\boxed{i\text{+}1}$ s.

\begin{rmk}Littlewood--Richardson tableaux are usually defined using the middle-eastern reading (which reads rows first and then concatenate them) but these are equivalent.
\end{rmk}
 
\subsection{Quantum groups and crystals}
The quantum groups we are interested in are deformations of universal enveloping algebras.

\begin{defi}
    Let $\gfrak$ be a Kac-moody Lie algebra associated to the Cartan matrix $A$ with simple roots and coroots $\pi=\{\alpha_i\}_{i\leq n}$ and $\pi^\vee=\{h_i\}_{i\leq n}$.
    The \emph{quantum group} $U_q(\gfrak)$ associated to $\gfrak$ as the $\CC(q)$-algebra generated by $e_i, f_i$ for $i\leq n$ and the formal symbols $q^h$ for $h\in \hfrak$ with the following relations
\[\begin{matrix}
    q^0=1  &\hspace{1cm}& q^{h}q^{h'}=q^{h+h'}\\
    \\
    q^h e_i=q^{-\alpha_i(h)h}e_i &\hspace{1cm}& q^hf_i=q^{\alpha_i(h)h}f_i\\
\end{matrix}\]
\[
    e_i f_j - f_j e_i =\delta_{ij}\dfrac{K_i-K_i^{-1}}{q_i-q_i^{-1}}\\
\]
and $q$-Serre relations for $i\neq j$
\[S_{ij}^+=\sum_{s\leq 1-a_{ij}}(-1)^s\qbinom[q_i]{1-a_{ij}}{s}e_i^{1-a_{ij}-s}e_je_i^s=0\]
$$S_{ij}^-=\sum_{s\leq 1-a_{ij}}(-1)^s\qbinom[q_i]{1-a_{ij}}{s}f_i^{1-a_{ij}-s}f_jf_i^s=0$$
where $q_i=q^{d_i}$ and $K_i=q^{d_ih_i}$ such that $D=(d_i)$ is the diagonal matrix symmetrizing $A$.
\end{defi}

The representation theory of quantum groups $U_q(\gfrak)$ for generic $q$ mirrors the one of Lie algebras (Chapter 5 \cite{Jantzen}). In particular, we have complete reducibility and classification of irreducible representations.
\begin{prop}\label{compred}
Every finite dimensional irreducible $U_q(\gfrak)$-representation $V$ can be written as a direct sum of irreducible representations.
\end{prop}

\begin{prop}\label{irrep}
The finite dimensional irreducible $U_q(\gfrak)$-representations are indexed by dominant integral weights $\lambda \in X^+$.
\end{prop}

There are two important specializations of $U_q(\gfrak)$, $q=1$ and $q=0$, which can't be done trivially since the quantum group relations are ill-defined. 
Instead, we define, respectively, the \emph{classical limit} (Section 3.4 \cite{hong_kang}) and the \emph{crystal limit} of $U_q(\gfrak)$-representation.
While the first one allows us to move between the quantum and classical representations, the second one provides \emph{crystals}, a powerful tool to study these representations.
Morally, crystals correspond a very well-behaved basis on $q=0$, composed of weight vectors and stable under the action of the standard generators.\\

To formalize this definition, we consider the ring $A_0=\{\tfrac{f}{g}\in \CC[q]\mid g(0)\neq 0\}$ of rational functions at $q=0$ and alter our standard operators, considering instead their weighted-analogs.

\begin{defi}
   The \emph{Kashiwara operators} $\e$ and $\f$ act on 
 \[v=\sum_{k<N}\dfrac{1}{[k]_q!}f^kv_k \htext{with} v_k \in ker(e_i)\]
 via
\[\e(v)=\sum_{k<N} \dfrac{1}{[k-1]_q!} f_i^{k-1}\; v_k \hspace{2cm} \f(v)=\sum_{k<N} \dfrac{1}{[k+1]_q!}f_i^{k+1}\; v_k\]
\end{defi}
\begin{rmk}
    This action well defined in any finite-dimensional representation due to complete reducibility (Lemma 4.1.1 in \cite{hong_kang}).
\end{rmk}

\begin{defi}
     A \emph{crystal basis} of a $U_q(\gfrak)$-representation $V$ is a pair $(\Lscr, \Bscr)$ where $\Lscr$ is a free $A_0$-submodule of $V$ such that
\begin{enumerate}[(L1)]
    \item $\Lscr$ is finitely generated over $\gfrak$ and generates $V$ over $\CC(q)$;
    \item $\Lscr=\bigoplus_{\lambda\in X} \Lscr_\lambda$ where $\Lscr_\lambda=\Lscr \cap V_\lambda$ and $X$ are the integral weights;
    \item $\Lscr$ is stable under $\f$ and $\e$.
\end{enumerate}
\newpage
\noindent and $\Bscr$ is a $\CC$-vector space basis of $\Lscr/q\Lscr$ satisfying
\begin{enumerate}[(B1)]
    \item $\Bscr=\bigcup_{\lambda} \Bscr_\lambda$ where $\Bscr_\lambda=\Bscr\cap \Lscr_\lambda/q\Lscr_\lambda$
    \item $\f\Bscr, \e\Bscr\subseteq \Bscr\cup \{0\}$
    \item For all $b, b'\in \Bscr$, $b'=\f b$ if and only if $\e b'=b$
\end{enumerate}
\end{defi}

We say two crystal bases are isomorphic if there is an $A_0$-isomorphism between the lattices that restricts to a bijection of the bases and commutes with the Kashiwara operators.

\begin{thr}[Existence and Uniqueness]
    Every finite dimensional $U_q(\gfrak)$-representation has a unique crystal basis $(\Lscr, \Bscr)$, up to isomorphism.
\end{thr}
\begin{corl}
    A $U_q(\gfrak)$-representation has a connected crystal graph if and only if it is irreducible.
\end{corl}

These bases provide a visual description of the representation as a colored oriented graph, called the \emph{crystal graph}, with nodes given by $\Bscr$ and edges given by
\[b\overset{i}\longrightarrow\f (b)\]

Moreover, crystals are compatible with the tensor product and the Kashiwara action in the product of two bases is described using functions
\[\varphi_i(b)=max_{n}\{\f^n b\neq 0\} \htext{and} \varepsilon_i(b)=max_{n}\{\e^n b\neq 0\}\]

\begin{thr}[Tensor Product Rule]\label{tprod}
    Let $V_i$ be $U_q(\gfrak)$-representations with crystal basis $(\Lscr_i, \Bscr_i)$ for $i=1, 2$. Then $V_1\otimes V_2$ has crystal basis $(\Lscr_1\otimes \Lscr_2,\; \Bscr_1\times\Bscr_2)$ and the Kashiwara operator act via
    \[\f(b_1\otimes b_2)= \left\{\begin{matrix}
    \f(b_1)\otimes b_2 & \text{if } \varphi_i(b_1)>\varepsilon_i(b_2)\\
    b_1\otimes \f(b_2) & \text{if } \varphi_i(b_1)\leq\varepsilon_i(b_2)
    \end{matrix}\right.\]
    \[\e(b_1\otimes b_2)= \left\{\begin{matrix}
    \e(b_1)\otimes b_2 & \text{if } \varphi_i(b_1)\geq\varepsilon_i(b_2)\\
    b_1\otimes \e(b_2) & \text{if } \varphi_i(b_1)<\varepsilon_i(b_2)
    \end{matrix}\right.\]
\end{thr}
\begin{rmk}
    Note that $\varphi_i(b)-\varepsilon_i(b)=\langle h_i, wt(b)\rangle$, as can be seen by restricting to the subalgebra $\gfrak_i\cong U_q(sl_2)$ generated by $e_i, f_i$ and $h_i$
\end{rmk}

For proofs and a detailed account of the constructions in this subsection we refer the reader to Chapters 4 and 5 of \cite{hong_kang}.

{\section{Restricting crystals}

To investigate the symplectic branching using crystals we need to describe the $U_q(sp_{2n})$-highest weight vectors in the crystals of the irreducible $U_q(gl_{2n})$-representations $\Bscr_{gl_{2n}}(\lambda)$.\\

Let $V$ be the natural representation of $U_q(gl_{2n})$ and $U_q(sp_{2n})$ with standard basis $\{v_i\}_{i\leq 2n}$.
Set $\Lscr= span_{A_0}\{v_i\}_{i\leq 2n}$ and define distinguished $\Lscr/q\Lscr$ elements
\[\boxed{\,i\,}=v_i+q\Lscr \htext{and} \boxed{\overline{\,i\,}}=(-1)^{n-i}v_{2n-i+1}+q\Lscr\]
such that we can write $\Lscr/q\Lscr = span_{\CC}[n] =span_{\CC}\Acal_n$ for the alphabet
\[\Acal_n=\{1<2<\cdots<n<\overline{n}<\cdots<\overline{2}<\overline{1}\}\]

We will use the following combinatorial model for the type $A$ crystals, as found in \cite{hong_kang}.

\begin{prop}\label{typeA}
    Let $\lambda\in X^+_{\ge 0}$ and $V_{gl_n}(\lambda)$ be the corresponding irreducible representation of $U_q(gl_{n})$. The associated crystal is given by 
    \[\Bscr_{gl_{n}}(\lambda)=\begin{Bmatrix}
\text{semistandard}\\\text{Young tableaux on }[n]\\ 
\text{with shape }\lambda
\end{Bmatrix}\]
    with edges $Y\overset{i}{\longrightarrow} Y'$ for all $\f(v_Y)=v_{Y'}$.
\end{prop}
\begin{corl}\label{hwtA}
    The only highest weight vector in $\Bscr_{gl_{n}}(\lambda)$ is $v_Y$ corresponding to the canonical tableau of shape $\lambda$ defined as $Y(i, j)=i$.
\end{corl}
\begin{proof}
    Since $V_{gl_n}(\lambda)$ is irreducible there is indeed only one highest weight vector in its crystal and we know $v_Y$ is this element because $wt(v_Y)=\lambda$.
\end{proof}

This model is endowed with a $U_q(sp_{2n})$-action coming from the embedding
\[\Bscr_{gl_{2n}}(\lambda) \lhook\joinrel\longrightarrow \pm\Bscr^{\otimes N}\]
where $\Bscr$ is the symplectic crystal of the natural representation, given by
\[\boxed{1}\overset{1}{\longrightarrow}\boxed{2}\overset{2}{\longrightarrow}\cdots \longrightarrow \boxed{n}\overset{n}{\longrightarrow}\boxed{\overline{n}}{\longrightarrow}\cdots \overset{2}{\longrightarrow}\boxed{\overline{2}}\overset{1}{\longrightarrow}\boxed{\overline{1}}\]
and weights
\[wt(\boxed{i})=\epsilon_i\hspace{1cm}wt(\boxed{\overline{\,i\,}})=-\epsilon_i\]

In the tableau, this embedding is just a translation of the fillings from $[2n]$ to $\Acal_n$.

\begin{rmk}
    There is no embedding of $U_q(sp_{2n})$ into $U_q(gl_{2n})$. Therefore, there is a priori no symplectic action on the type $A$ crystal. The model we use for the crystals corresponds to a choice of mapping $V_{gl_{2n}}(\lambda) \hookrightarrow V^{\otimes N}$ and, thus,  implicitly chooses a symplectic action. These choices are not trivial. In fact, Naito--Sagaki's conjecture does not hold for all copies of $V_{gl_{2n}}(\lambda)$ in $V^{\otimes N}$.
\end{rmk}

In this framework, we can reformulate the tensor product rule (\Cref{tprod}).
\begin{lemma}\label{hwtC}
    An element $v=a_1\otimes \cdots \otimes a_N \in \Bscr^{\otimes N}$ is of $U_q(sp_{2n})$-highest weight if and only if
    \[\hspace{1.5cm}\#\{a_j=i \text{ or } \overline{i\!+\!1}\mid j< k\}\ge \# \{a_j=i\!+\!1\text{ or } \overline{\;i\;}\mid j\leq k\}\hspace{0.8cm}\forall i\leq n,\; k\leq N\]
\end{lemma}
\begin{proof}
    Write $v=v'\otimes a_N$ and recall $wt(\boxed{i})=\epsilon_i$ while $wt(\boxed{\overline{i}})=-\epsilon_i$.\\
    By the tensor product rule , $\e(v)=0$ if and only if
    $\e(v')=0$ and $\varphi_i(v')\ge \varepsilon(a_N)$ or $\e(a_N)=0$  and $\varphi_i(v')< \varepsilon(a_N)$. However, the last case is absurd, as it implies $\varphi_i(v')< 0$.\\
    Assume, by induction on $N$, that the result holds for $v'$. Then, $\e(v')=0$ is equivalent to the above equations for $k\leq N-1$. Moreover, 
    \begin{align*}
        \varphi_i(v')-\varepsilon_i(a_N)&= \langle h_i, wt(v')\rangle - \varepsilon_i(a_N)\\
        &= \#\{a_j=i\text{ or }\overline{\;i\text{+}1\;}\mid j\leq N-1\}- \# \{a_j=i\text{+}1 \text{ or }\overline{\;i\;} \mid j\leq N-1\}\\
        &\hspace{7.3cm}- \#\{a_N=i\text{+}1\text{ or }\overline{\;i\;}\}
    \end{align*}
    which makes $ \varphi_i(v')\geq \varepsilon_i(a_N)$ equivalent to the remaining equation.
\end{proof}

It is then possible to describe the relevant tableaux in more concrete terms.
\begin{defi}
    We say $Y$ is a \emph{$\mathbf{sp_{2n}}$-highest weight tableaux} if its reading $v_Y$ is a $U_q(sp_{2n})$-highest weight vector.
\end{defi}
\begin{lemma}\label{sptab}
    If a semistandard Young tableau $Y$ on $\Acal_n$ is a $sp_{2n}$-highest weight tableau, then all the entries $\boxed{i}$ with positive weights are in the corresponding $i$-th row. Moreover, every entry with negative weight $\boxed{\overline{\,i\,}}$ appears after the $i$-th row.
\end{lemma}}
\begin{proof}
    Note that before the first entry with negative weight we are restricted to the type A case and by \Cref{hwtA} these entries match the row number. This means that if $Y$ has no negative weights we are done. Otherwise, let $a_k=\overline{\,i\,}$ be the first negative weight in the reading of $Y$.
    By the highest weight condition (and \Cref{hwtC}) we know that
    \[\#\{a_j=i\mid j< k\}\ge \# \{a_j=i\!+\!1\mid j\leq k\}+1\]
    which means there must an $i$ filling in an outer corner. Therefore, by deleting this pair of entries $\{i, \overline{\,i\,}\}$ and sliding the column with $\overline{\,i\,}$ up, if necessary, we obtain a Young tableau $Y'$. Additionally, due to \Cref{hwtC}, we know $Y'$ is also of highest weight, because the fillings $i$ and $\overline{\,i\,}$ cancel each other out in every prefix after this point. By induction on the shape of the tableau, all entries in $Y'$ satisfy the conditions of the lemma and so does the deleted pair, by which are done.
\end{proof}

One way of interpreting this result is every semistandard Young tableaux $Y$ of shape $\lambda$ and weight $\mu$ can be split into a canonical part of shape $\mu$ and a semistandard part of shape $\lambda \setminus \mu$ filled by ordered $\{i, \overline{\,i\,}\}$-pairs.
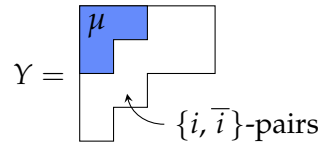
\[\begin{tikzpicture}[scale=0.45]
    \draw (0,0) -- (0,4) -- (4,4) -- (4,2) -- (2,2) -- (2,1) -- (1,1) -- (1,0) -- cycle;
    \filldraw[fill=blue!75] (0,2) -- (0,4) -- (2,4) -- (2,3) -- (1,3) -- (1,2) -- cycle;
    \node[anchor=center] at (-1.2,2){$Y=$};
    \node[anchor=center] at (0.5,3.4) {\large{$\mu$}};
    \node[anchor=west] at (2.5,0.5) {$\{i, \overline{\,i\,}\}$-pairs};
    \draw[-stealth] (2.5,0.5) to[bend left=40] (1.4,1.4);
\end{tikzpicture}\]
This restricts the question to understanding the skew-tableaux part.

\section{Cascading operation}

Before we can describe the symplectic branching bijection, we need to define an operation which plays a vital part in it and formulate some of its properties.
\begin{defi}
Let $s=(1, \cdots, m)$ be an ordered sequence of fillings in a tableaux $Y$.\\ The operation of \emph{cascading the sequence s} is defined as follows:
\begin{itemize}
    \item Delete $\boxed{1}$\;;
    \item Replace remaining fillings in the sequence via $i\mapsto i-1$.
\end{itemize}
\end{defi}

\begin{rmk}
    The first filling becomes a hole, so the cascaded tableaux are not necessarily Young tableaux or skew-tableaux.
\end{rmk}

\begin{ex}\label{ex:3cascades}
    Consider this tableau with 3 different sequences, which cascade as follows
    
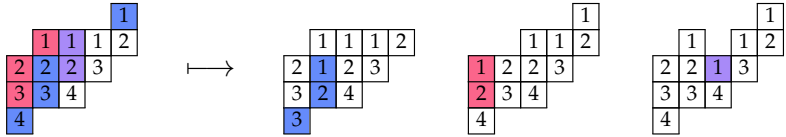
\[\begin{ytableau}\none & \none & \none & \none &*(blue!75)1\\
    \none &  *(red!75)1 & *(purple!50)1 & 1 & 2\\
    *(red!75)2 & *(blue!75)2& *(purple!50)2 &3\\
    *(red!75)3&*(blue!75)3 & 4\\
    *(blue!75)4
    \end{ytableau}
    \hspace{15pt}\longmapsto \hspace{15pt}
    \begin{ytableau}\none & \none & \none & \none\\
    \none &  1 & 1 & 1 & 2\\
    2 & *(blue!75)1& 2 & 3\\
    3&*(blue!75)2 & 4\\
    *(blue!75)3
    \end{ytableau}
    \hspace{20pt}
    \begin{ytableau}\none & \none & \none & \none & 1\\
    \none &  \none & 1 & 1 &  2\\
    *(red!75)1 & 2& 2 & 3\\
    *(red!75)2&3 & 4\\
    4
    \end{ytableau}
    \hspace{20pt}
    \begin{ytableau}\none & \none & \none & \none & 1\\
    \none &  1 & \none & 1 & 2\\
    2 & 2& *(purple!50)1 & 3\\
    3&3 & 4\\
    4
    \end{ytableau}\]
\end{ex}

It is necessary to establish some conditions on the sequence to guarantee that the cascade will preserve basic tableau properties.
\begin{lemma}\label{nicecascade}
    Let $s=(1, \cdots, m)$ be a sequence such that:
    \begin{enumerate}[label=(\Roman*)]
        \item $1\in s$ is left-most filling in its row;
        \item There is no $i$ between $i$ and $i+1\in s$ in $v_Y$;
        \item $m\in s$ is the last $m$ filling in $v_Y$;
    \end{enumerate}
    The tableau $Y$ is semistandard if and only if the cascaded tableau $Y'$ fulfills the semistandard condition. Moreover, if the weight of $Y'$ corresponds to a partition, then $Y$ is Littlewood--Richardson if and only if the cascaded tableau $Y'$ fulfills the Littlewood--Richardson condition.
\end{lemma}
\begin{proof}
For semistandardness, consider
 \[\adjustbox{scale=1}{$ \begin{array}{c|c|c}
 & A &  \\ \hline
B & \;i\; & C \\ \hline
 & D &  
\end{array}\;\; {\longrightarrow} \;\; \begin{array}{c|c|c}
 & A' &  \\ \hline
B' & i\text{-}1 & C' \\ \hline
 & D' &  
\end{array}$}\]

Assuming $Y$ is semistandard, we only need to consider the cases where $A=i-1$ or $B=i$. 
Here (II) is the crucial condition. If $A=i-1$ then $A\in s$ and $A'=i-2$. Meanwhile if $B=i$ then we must have $D=i+1 \in s$ and, by semistandardness of $Y$, the filling to its left is also $i+1$. Thus, we can apply the same logic again and again to conclude we have 
\[\adjustbox{scale=0.9}{$
\begin{array}{c|c|c}
 &  &  \\ \hline
 i& i & \;\; \\ \hline
 i\text{+}1&i\text{+}1 &  
 \\\vdots&\vdots\\
 m&m
\end{array}
\;\; {\longrightarrow} \;\;
\begin{array}{c|c|c}
 &  &  \\ \hline
 i& i\text{-} 1& \;\; \\ \hline
 i\text{+}1&i & 
 \\\vdots&\vdots\\
 m&m\text{-}1
\end{array}
$}\]
which contradicts (III).\newpage

Now we prove the remaining implication. Assuming that the cascaded tableau $Y'$ is semistandard, we only need to consider the cases where $D'=i$ or $C'=i-1$. 
This works in an analogous way. If $D=i$ that would contradict (II) so $D'=i$ implies $D=i+1\in s$. Meanwhile, if $C'=i-1$ then $C=i-1\notin s$, so $A=i-1\in s$ and, by semistandardness, we would be able to repeat the process and obtain
\[\adjustbox{scale=0.9}{$
\begin{array}{c|c|c}
  2&1
 \\\vdots&\vdots\\
 i\text{-}1 & i\text{-}2 & \;\; \\ \hline
 i& i\text{-}1 & \;\; \\ \hline
  &  &  \\ 
\end{array}
\;\; {\longrightarrow} \;\;
\begin{array}{c|c|c}
  1&\ast
 \\\vdots&\vdots\\
  i\text{-}2 & i\text{-}2 & \;\; \\ \hline
 i\text{-}1& i\text{-}1 & \;\; \\ \hline
  &  &  \\ 
\end{array}
$}\]
which contradicts (I).\\

For the Littlewood--Richardson condition, let $pr(Y, A)$ be the weight of the prefix of $v_Y$ until a box $A$. We will describe them case-by-case. If $A$ appears before $1\in s$, then $pr(Y', A)=pr(Y, A)$. If $A\in s$, then $pr(Y', A)=pr(Y, A')$, where $A'$ is the box immediatly before $A$. Meanwhile, if $A$ is strictly between $i$ and $i+1 \in s$, then $pr(Y', A)$ only has one less $i$ than $pr(Y,A)$, i.e. $pr(Y', A)_j=pr(Y, A)_j$ for $j\neq i$ and $pr(Y', A)_i=pr(Y, A)_i-1$. However, due to $(II)$,
\begin{align*}
pr(Y', A)_i=pr(Y', \,i\!+\!1\in s)_i\ge pr(Y', \,i\!+\!1\in s)_{i+1}\ge pr(Y', A)_{i+1}
\end{align*}

Lastly, if $A$ appears after $m\in s$ we have can use this same logic due to $(III)$ and the fact that $wt(v_{Y'})=wt(Y')$ corresponds to a partition.

Thus, the changes preserve the correspondence with a partition.
\end{proof}
\vbox{\begin{ex}
    In \Cref{ex:3cascades} Y is a LR tableau and we have, respectively:
    \begin{itemize}
        \item A sequence that does not satisfy $(II)$, resulting in a tableau that is non semistandard and does not satisfy the LR-condition;
        \item A sequence satisfying $(I)-(III)$, resulting in a LR tableau;
        \item A sequence that does not satisfy $(I)$ or $(III)$, resulting in a cascaded tableau that is not a tableau and not semistandard but still satisfies the LR-condition.
    \end{itemize}
\end{ex}}

Additionally, sequences as described above don't cross each other so there is a direct relation between where they end and start, as described in the following lemma.

\begin{lemma}\label{cascades1}
    Let $Y_1$ and $Y_2$ be tableaux with sequences $s_1=(1, \cdots \,, m_1)$ and $s_2=(1, \cdots \,, m_2)$ satisfying $(I)-(III)$ and let $Y_2$ be obtained from $Y_1$ by cascading $s_1$.\\
    If $m_1\leq m_2$, then $k\in s_1$ appear strictly after $k \in s_2$ in the respective readings, for all $k\leq m_1, m_2$.\\
    If $m_1> m_2$, then $k\in s_1$ appear (non-strictly) before $k \in s_2$ in the readings, for all $k\leq m_1, m_2$.
\end{lemma}
\begin{proof}
    Let $m=min\{m_1, m_2\}$ such that, by condition (III) on $s_1$ and $s_2$, we have that $m\in s_1$ appears after $m\in s_2$ exactly when $m_1\leq m_2$.\\
    
    The argument is then an induction on $m-k$.
    
    If $k+1\in s_1$ appears after $k+1\in s_2$ then $k\in s_1$ appears after $k\in s_2$. Otherwise, $k \in s_2$ is between $k+1$ and $k \in s_1$, contradicting condition (II) on $s_1$. Also note that here $k+1\in s_1$ and $k \in s_2$ cannot coincide, since both $k\in s_1$ and $k+1 \in s_2$ are between them.
    
    Meanwhile, if $k+1\in s_1$ appears before $k+1\in s_2$ then $k\in s_1$ appears before $k\in s_2$. Otherwise, $k+1 \in s_1$, which corresponds to $k\in Y_2$, is between $k$ and $k+1 \in s_2$, contradicting condition (II) on $s_2$.
\end{proof}

\begin{ex}
Consider the following consecutive cascades of sequences satisfying $(I)-(III)$
\vspace{-5pt}
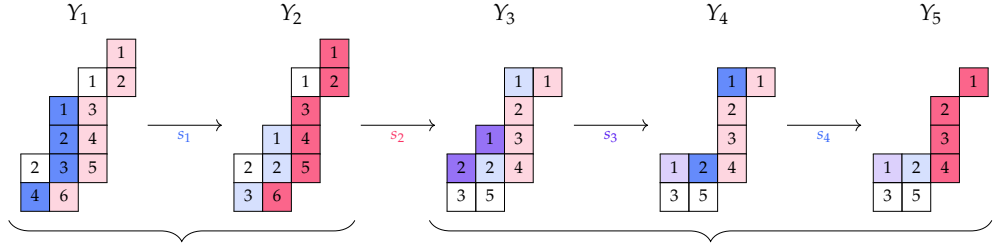
\[\adjustbox{scale=0.8,center}{\ctableau
\begin{tikzcd}[row sep = -0.2em, column sep = 1em, ampersand replacement=\&]
    Y_1\& \& Y_2 \& \& Y_3 \& \& Y_4\& \& Y_5\\
    \begin{ytableau}
    \none & \none & \none & *(red!20) 1 \\
    \none &  \none & 1 & *(red!20) 2 \\
    \none &  *(blue!75)  1  & *(red!20) 3 \\
    \none & *(blue!75)  2 & *(red!20)4\\
    2 & *(blue!75)  3 & *(red!20) 5 \\
    *(blue!75) 4 & *(red!20) 6
    \end{ytableau}  
    \arrow[rr, "s_1"'{color=blue}] \&\;\& 
    \begin{ytableau}
    \none & \none & \none & *(red!75) 1 \\
    \none &  \none & 1 & *(red!75) 2 \\
    \none &  \none & *(red!75) 3 \\
    \none & *(blue!20)  1 & *(red!75)4\\
    2 & *(blue!20)  2 & *(red!75) 5 \\
    *(blue!20) 3 & *(red!75) 6
    \end{ytableau}  
    \arrow[rr,"s_2"'{color=red}] \&\;\&
    \begin{ytableau}
    \none & \none & \none & \none \\
    \none &  \none & *(blue!20) 1 & *(red!20) 1 \\
    \none &  \none & *(red!20) 2 \\
    \none & *(purple!60)  1 & *(red!20)3\\
    *(purple!60) 2 & *(blue!20)  2 & *(red!20) 4 \\
     3 &  5
    \end{ytableau} 
     \arrow[rr, "s_3"'{color=purple}] \&\;\&
    \begin{ytableau}
    \none & \none & \none & \none \\
    \none &  \none & *(blue!75) 1 & *(red!20) 1 \\
    \none &  \none & *(red!20) 2 \\
    \none & \none & *(red!20)3\\
    *(purple!20) 1 & *(blue!75)  2 & *(red!20) 4 \\
     3 &  5
    \end{ytableau} 
     \arrow[rr, "s_4"'{color=blue}] \&\;\&
    \begin{ytableau}
    \none & \none & \none & \none \\
    \none &  \none & \none & *(red!75) 1 \\
    \none &  \none & *(red!75) 2 \\
    \none & \none & *(red!75)3\\
    *(purple!20) 1 & *(blue!20)  2 & *(red!75) 4 \\
     3 &  5
    \end{ytableau}
    \\
    \arrow[rrr, start anchor=west, end anchor=east, no head, xshift=-2.5em, decorate, decoration={brace, mirror, amplitude=10pt}] 
    \&\&\&\;
    \arrow[rrrrr, start anchor=west, end anchor=east, no head, xshift=2em, decorate, decoration={brace, mirror, amplitude=10pt}] 
    \&\&\&\&\&\;
    \\
\end{tikzcd}}\]

In particular, while $m_i$ increases the sequences are disjoint, as the light shading highlights.\\
In contrast, note that $s_1$ and $s_3$ have a box in common. Thus, the cascading of these sequences interfere with each other, such that $1\in s_3$ does not correspond to $1$ in $Y_1$ while $2\in s_1$ does not correspond to $1$ in $Y_4$ and $Y_5$.
\end{ex}

\begin{lemma}\label{cascades2}
     Let $Y_i$ be tableaux with sequences $s_i=(1, \cdots \,, m_i)$ satisfying $(I)-(III)$ and let $Y_{i+1}$ be obtained by cascading $s_i$. Assume further that there is a $r$ such that
     \[m_1\leq \cdots \leq m_{r-1} > m_{r}\leq \cdots \leq m_{2r-1}\]
    Then, for fixed $r_0<r$, we have $m_i>m_{i+r}$ for all $i\leq r_0$ if and only if $k\in s_i$ appears strictly before $k \in s_{i+r}$ for all $k\leq m_i, m_{i+r}$ and $i< r_0$.
\end{lemma}
\begin{proof}
This argument follows the previous proof, however now we must consider that the cascadings might interfere with $s_i$ and $s_{i+r}$. Here it is important to note that $\{s_1, \cdots, s_{r-1}\}$ and well as $\{s_{r}, \cdots, s_{2r-1}\}$ are sequences with no boxes in common, thus this interference happens at most once per box.

$(\Leftarrow)$ We prove the contrapositive. Let $m_i\leq m_{i+r}$, then $m_i\in s_{i+r}$ comes from a $m_i\in Y_i$ or is $m_i+1\in s_{j}$ for some $i\leq j<r$. Either way, $m_i\in s_i$ appears strictly after $m_i\in s_{i+r}$.

$(\Rightarrow)$ Let $m_i>m_{i+r}$. Similarly to before, we have that $m_{i+r}\in s_{i}$ is sent to a $m_{i+r}\in Y_i$ or we have $m_{i+r}-1\in s_{j}$ for some $r\leq j<i+r$ --- either way it strictly precedes $m_{i+r}\in s_{i+r}$. Now assume $k+1\in s_i$ appears strictly before $k+1\in s_{r+i}$ and $k \in s_{i}$ appears after $k\in s_{i+r}$. This means that $k+1\in s_{i}$ is in between $k$ and $k+1\in s_{i+r}$. Due to $(II)$, this box cannot be filled by a $k$ in $Y_{i+r}$, so this must be the same box as $k\in s_{j}$ for some $j<i$. However, by induction on $i$, we have that $k\in s_j$ that appears before $k\in s_{j+r}$ which, by \Cref{cascades1}, appears before $k\in s_{i+r}$. Hence, we have a contradiction, which allows us to conclude that we have the desired order for all $k\leq m_{i+r}$.
\end{proof}

\section{Symplectic branching bijection}

We first focus on the case where $\ell(\lambda)\leq n$. Here, every Littlewood--Richardson tableau satisfies the $n$-symplectic condition on Sundaram's branching rule. Moreover, in this case, the highest weight condition is independent of $n$ (\Cref{hwtC} and \Cref{sptab}), which motivates the following definition.

\begin{defi}
    We say $Y$ is a \emph{$\mathbf{sp}$-highest weight tableau} if its reading $v_Y$ is a $U_q(sp_{2n})$-highest weight vector for $n$ is arbitrarily big. 
\end{defi}

We will construct a bijection $F$ between highest weight and Littlewood--Richardson tableaux inductively on $\Delta=|\lambda|-|\mu|$.
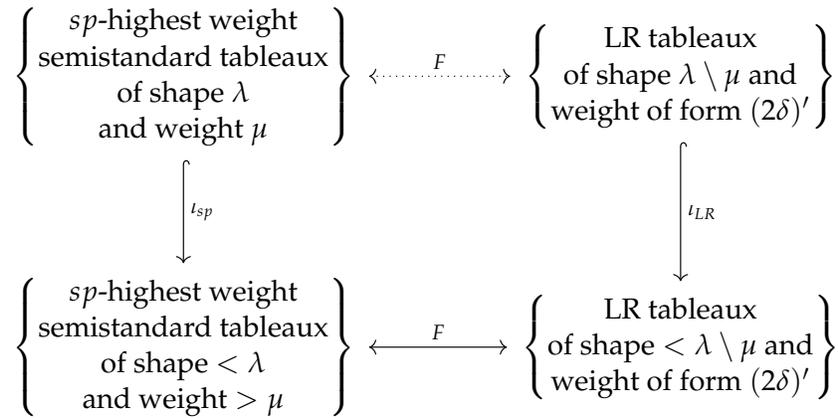
\[\begin{tikzcd}[ampersand replacement=\&]
        \begin{Bmatrix}
        \text{$sp$-highest weight}\\
        \text{semistandard tableaux}\\ 
        \text{of shape $\lambda$}\\ 
        \text{and weight $\mu$}
        \end{Bmatrix}
    \arrow[leftrightarrow, dotted, "F"]{rr}
    \arrow[hookrightarrow, "\iota_{sp}"]{dd}
    \&\&
        \begin{Bmatrix}
        \text{LR tableaux}\\ 
        \text{of shape $\lambda \setminus \mu$ and}\\ 
        \text{weight of form $(2\delta)'$}
        \end{Bmatrix}
    \arrow[hookrightarrow, "\iota_{LR}"]{dd}
    \\
    \\
        \begin{Bmatrix}
        \text{$sp$-highest weight}\\
        \text{semistandard tableaux}\\ 
        \text{of shape $<\lambda$}\\ 
        \text{and weight $>\mu$}
        \end{Bmatrix}
    \arrow[leftrightarrow, "F"]{rr}
    \&\&
        \begin{Bmatrix}
        \text{LR tableaux}\\ 
        \text{of shape $<\lambda \setminus \mu$ and}\\ 
        \text{weight of form $(2\delta)'$}
        \end{Bmatrix}
\end{tikzcd}\]

\vbox{\begin{defi}
    The inclusion $\iota_{sp}$ is obtained through the following process:
\begin{itemize}
    \item Delete last entry with negative weight;
    \item Slide last row to the left if needed to obtain a Young tableau, i.e. if $\lambda_{\ell(\lambda)}>1$.
\end{itemize}
\end{defi}}

\begin{ex} Consider the following tableau and $\iota_{sp}$ action
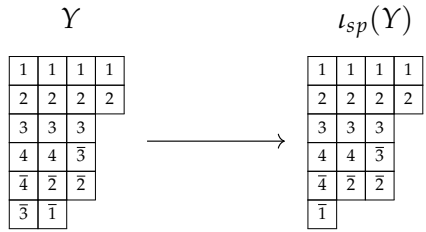
\[\ctableau
\begin{tikzcd}[row sep = -0.2em, ampersand replacement=\&]
Y\&\&\hspace{1em}\iota_{sp}(Y)\\
    \adjustbox{scale=0.8}{\begin{ytableau}
    1 & 1 & 1 & 1 \\
    2 & 2 & 2 & 2 \\
    3 & 3 & 3 \\
    4 & 4 & {\overline{3}} \\
    {\overline{4}} & {\overline{2}} & {\overline{2}} \\
    {\overline{3}} & {\overline{1}} \\
    \end{ytableau}}
\arrow[rr]
\;\&\&\;
    \adjustbox{scale=0.8}{\begin{ytableau}
    1 & 1 & 1 & 1 \\
    2 & 2 & 2 & 2 \\
    3 & 3 & 3 \\
    4 & 4 & {\overline{3}} \\
    {\overline{4}} & {\overline{2}} & {\overline{2}} \\
    {\overline{1}} \\
    \end{ytableau}}
\end{tikzcd}\]
which, when restricted to $\lambda\setminus\mu$, gives us the deletion of a $\{i,\overline{i}\}$-pair, as highlighted below
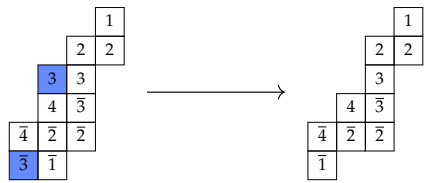
\[
\begin{tikzcd}[row sep = 0.2em, ampersand replacement=\&]
    \adjustbox{scale=0.8}{\begin{ytableau}
    \none & \none & \none & 1 \\
    \none &  \none & 2 & 2 \\
    \none & *(blue!75)3 & 3 \\
    \none & 4 & {\overline{3}} \\
    {\overline{4}} & {\overline{2}} & {\overline{2}} \\
    *(blue!75){\overline{3}} & {\overline{1}} \\
    \end{ytableau}}
\arrow[rr]
\;\&\&\;
    \adjustbox{scale=0.8}{\begin{ytableau}
    \none & \none & \none & 1 \\
    \none &  \none & 2 & 2 \\
    \none & \none & 3 \\
    \none & 4 & {\overline{3}} \\
    {\overline{4}} & {\overline{2}} & {\overline{2}} \\
    {\overline{1}} \\
    \end{ytableau}}
\end{tikzcd}\]
\end{ex}
\newpage

\begin{lemma}\label{iC}
    The map $\iota_{sp}$ is well-defined and injective for fixed $(\lambda, \mu)$.
\end{lemma}
\begin{proof}
    The image is clearly semistandard. To see that it is of highest weight, note that the deletion doesn't change any prefixes and the sliding changes them by at most one $\overline{i}$ from the last row. However, using the description of $sp$-highest weight tableaux (\Cref{sptab}) we can see that this change is not a problem, since every $i+1$ is directly preceded by $i$ and since, by semistandardness, there is no $\overline{i\,}$ or $\overline{i\text{+}1}$ filling between the positions of a sliding filling $\overline{i\,}$.
    Moreover, we can identify the deleted filling from the weight of $\iota_{sp}(Y)$ and $\mu$, so we can recover $Y$.
\end{proof}

\begin{defi}
    The inclusion $\iota_{LR}$ is obtained through the following process
\begin{itemize}
    \item Locate the last ordered sequence $s=(1, 2, \cdots, m)$ of fillings in $Y$, according to the anti-lexicographic order on the reading (that is, such that $m$ is the last filling and there are no $i$ between $i$ and $i+1 \in s$);
    \item Cascade $s$;
    \item Delete the last filling;
    \item Slide last row to the left if needed to obtain a Young tableau, i.e. if $\lambda_{\ell(\lambda)}>1$.
\end{itemize}
\end{defi}

\begin{ex} Consider the following tableau and $\iota_{LR}$ action\\
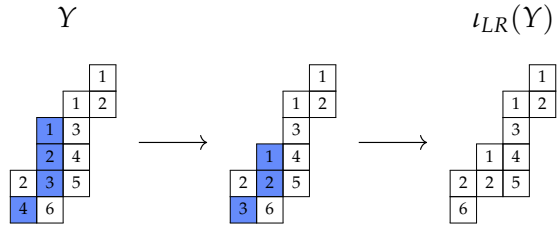
\[\ytableausetup{smalltableaux, onlyboxsize=1.1em}
\begin{tikzcd}[row sep = 0.2em, ampersand replacement=\&]
Y\&\&\hspace{1em}\iota_{LR}(Y)\\
    \adjustbox{scale=0.8}{\begin{ytableau}
    \none & \none & \none & 1 \\
    \none &  \none & 1 & 2 \\
    \none &  *(blue!75) 1 & 3 \\
    \none & *(blue!75)  2 &4\\
    2 & *(blue!75)  3 & 5 \\
    *(blue!75) 4 & 6 \\
    \end{ytableau}}
\arrow[r]
\;\& \;
    \adjustbox{scale=0.8}{\begin{ytableau}
    \none & \none & \none & 1 \\
    \none &  \none & 1 & 2 \\
    \none &  \none & 3 \\
    \none & *(blue!75)  1 &4\\
    2 & *(blue!75)  2 & 5 \\
    *(blue!75)  3 & 6 \\
    \end{ytableau}} 
\arrow[r]
\;\&\;
    \adjustbox{scale=0.8}{\begin{ytableau}
    \none & \none & \none & 1 \\
    \none &  \none & 1 & 2 \\
    \none &  \none & 3 \\
    \none &  1 &4\\
    2 &  2 & 5 \\
    6 \\
    \end{ytableau}}
\end{tikzcd}\]
where the sequence $s$ is highlighted.
\end{ex}

\begin{lemma}\label{iA}
    The map $\iota_{LR}$ is well-defined and injective for fixed $(\lambda, \mu)$.
\end{lemma}
\begin{proof}
First note that a sequence as desired always exists because by the highest weight condition there is always an $i-1$ before $i$ and note that it satisfies the conditions $(I)-(III)$ of \Cref{nicecascade}.

Hence, the image exists and that the cascaded tableau is Littlewood–Richardson. Meanwhile, we can see that the deletion and sliding will also preserve these conditions with a simplified version of the argument we used in \Cref{iC}.

We can also check that the image has weight and shape of the desired form
\[wt(\iota_{LR}(Y))=wt(Y)-\epsilon_m-\epsilon_{m-1}
\hspace{1cm}sh(\iota_{LR}(Y))=(\lambda-\epsilon_{\ell(\lambda)})\setminus (\mu+ \epsilon_{i})\]
where $i$ is the row of $1\in s$.

We are then left to show injectivity. Note that from the shape of $\iota_{LR}(Y)$ and $\lambda \setminus \mu$ we can tell which box was deleted. And, since the original sequence didn't have any $i$ between $i$ and $i+1$, its image won't have any  $i$ between $i-1$ and $i$. Therefore, by choosing the first sequence $1, 2, \cdots, k$ after the deleted box, according to the lexicographic order on the reading, we recover the remaining part of the original sequence. Hence, $Y$ can be reconstructed by sliding, adding $1$ and $m$ back and mapping $i\mapsto i+1$.
\end{proof}

We can then establish the connection between these inclusions and use it to construct our bijection $F$.

\begin{defi}
We define $F$ piecewise as
\[F_{\lambda, \mu}:
        \begin{Bmatrix}
        \text{ LR tableaux}\\ 
        \text{of shape $\lambda \setminus \mu$ and}\\ 
        \text{weight of form $(2\delta)'$}
        \end{Bmatrix}
    \longrightarrow 
        \begin{Bmatrix}
        \text{$sp$-highest weight}\\
        \text{semistandard tableaux}\\ 
        \text{of shape $\lambda$}\\ 
        \text{and weight $\mu$}
        \end{Bmatrix}
        \]
where
\[F_{\lambda, \lambda}:\hspace{0.2cm} 
        \begin{matrix}
        \text{empty tableau}\\
        Y(\emptyset)=\emptyset
        \end{matrix}
    \hspace{0.2cm}\longmapsto\hspace{0.2cm}
        \begin{matrix}
        \text{canonical tableau}\\
        Y(i,j)=i
        \end{matrix}\]
and, recursively,
\[F_{\lambda, \mu}(Y)= \iota_{sp}^{-1} \circ F\circ \iota_{LR}(Y)\]
\end{defi}

\begin{ex}
Consider the tableaux from the previous examples and their correspondence, illustrated below with the canonical $\mu$ fillings omitted\\

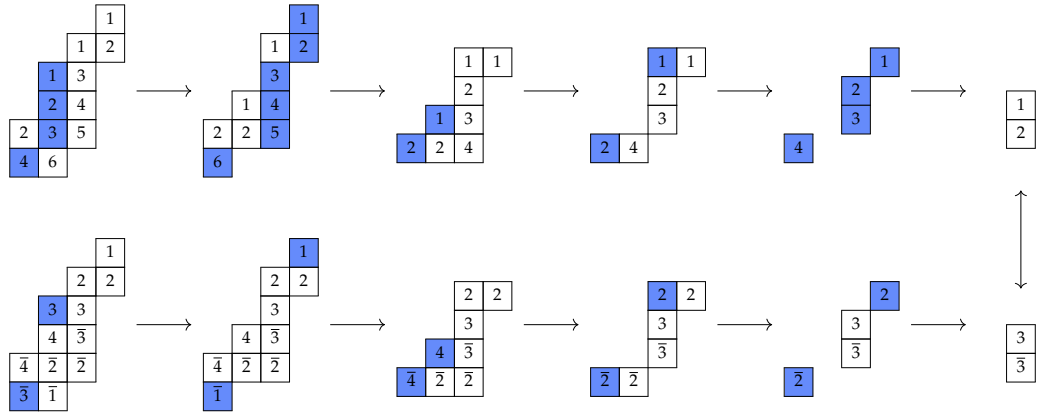
\adjustbox{scale=0.8,center}{
    \begin{tikzcd}[ampersand replacement=\&]
    \ctableau
    \begin{ytableau}
    \none & \none & \none & 1 \\
    \none &  \none & 1 & 2 \\
    \none &  *(blue!75) 1 & 3 \\
    \none & *(blue!75)  2 &4\\
    2 & *(blue!75)  3 & 5 \\
    *(blue!75) 4 & 6 \\
    \end{ytableau} 
    \arrow[r]\&
    \begin{ytableau}
    \none & \none & \none & *(blue!75) 1 \\
    \none &  \none & 1 & *(blue!75) 2 \\
    \none &  \none & *(blue!75) 3 \\
    \none & 1 & *(blue!75) 4\\
    2 & 2 & *(blue!75) 5 \\
    *(blue!75) 6 \\
    \end{ytableau}
    \arrow[r]\&
    \begin{ytableau}
    \none & \none & \none & \none\\
    \none &  \none & 1 & 1 \\
    \none &  \none & 2 \\
    \none & *(blue!75) 1 & 3\\
    *(blue!75) 2 & 2 & 4 \\
    \end{ytableau}
    \arrow[r]\&
    \begin{ytableau}
    \none & \none & \none & \none\\
    \none &  \none & *(blue!75) 1 & 1 \\
    \none &  \none & 2 \\
    \none & \none & 3\\
     *(blue!75) 2 & 4 \\
    \end{ytableau} 
    \arrow[r]\&
    \begin{ytableau}
    \none & \none & \none & \none\\
    \none &  \none & \none & *(blue!75) 1 \\
    \none &  \none & *(blue!75) 2 \\
    \none & \none & *(blue!75) 3\\
     *(blue!75) 4 \\
    \end{ytableau} 
    \arrow[r]\&
    \begin{ytableau}
    \none & \none \\
    \none &  \none \\
    \none &  1 &  \none\\
    \none &  2 &  \none\\
    \end{ytableau}
    \arrow[d, leftrightarrow, yshift=-15]\\
    \begin{ytableau}
    \none & \none & \none & 1 \\
    \none &  \none & 2 & 2 \\
    \none & *(blue!75) 3 & 3 \\
    \none & 4 & {\overline{3}} \\
    {\overline{4}} & {\overline{2}} & {\overline{2}} \\
    *(blue!75) {\overline{3}} & {\overline{1}} \\
    \end{ytableau} 
    \arrow[r]\&
    \begin{ytableau}
    \none & \none & \none & *(blue!75) 1 \\
    \none &  \none & 2 & 2 \\
    \none &  \none & 3 \\
    \none & 4 & {\overline{3}} \\
    {\overline{4}} & {\overline{2}} & {\overline{2}} \\
    *(blue!75) {\overline{1}} \\
    \end{ytableau}
    \arrow[r]\&
    \begin{ytableau}
    \none & \none & \none & \none\\
    \none &  \none & 2 & 2 \\
    \none &  \none & 3 \\
    \none & *(blue!75) 4 & {\overline{3}} \\
    *(blue!75) {\overline{4}} & {\overline{2}} & {\overline{2}} \\
    \end{ytableau}
    \arrow[r]\&
    \begin{ytableau}
    \none & \none & \none & \none\\
    \none &  \none & *(blue!75) 2 & 2 \\
    \none &  \none & 3 \\
    \none & \none & {\overline{3}} \\
     *(blue!75) {\overline{2}} & {\overline{2}} \\
    \end{ytableau} 
    \arrow[r]\&
    \begin{ytableau}
    \none & \none & \none & \none\\
    \none &  \none & \none & *(blue!75) 2 \\
    \none &  \none & 3 \\
    \none & \none & {\overline{3}}\\
     *(blue!75) {\overline{2}} \\
    \end{ytableau} 
    \arrow[r]\&
    \begin{ytableau}
    \none & \none \\
    \none &  \none \\
    \none &  3 &  \none\\
    \none &  {\overline{3}} &  \none\\
    \end{ytableau}\\
    \end{tikzcd}}
Note that the negative filling deleted by $\iota_{sp}$ corresponds the row number of the $1$ in the sequence altered by $\iota_{LR}$.
\end{ex}

We are now ready to prove the core of our main result.
\begin{prop}\label{Fgood}
    The map $F$ is a well-defined bijection.
\end{prop}
\begin{proof}
We need to show that $F$ is a well-defined bijection, that is that $F(Im\;\iota_{LR}) = Im\;\iota_{sp}$.\\

Since these functions are injective, we know they have partial inverses. That means that, given a Littlewood--Richardson tableau $Y$ and the corresponding $sp$-highest weight tableau $Y'=F(Y)$, we can construct tentative preimages $Y_0$ and $Y_0'$ with same $(\lambda, \mu)$ such that $Y=\iota_{LR}(Y_0)$ and $Y'=\iota_{sp}(Y_0')$ precisely when $Y_0$ and $Y_0'$ are in the respective domains. 
Hence, our aim is to show that these are equivalent, i.e. $Y_0\in Dom(\iota_{LR})\Leftrightarrow Y'_0\in Dom(\iota_{sp})$, and it follows that $F(Y_0)=Y_0'$.\\

First, note that by \Cref{nicecascade} reversing the cascade preserves the semistandard and LR-conditions.
Moreover, if $Y_0$ and $Y_0'$ are semistandard, then the sliding does nothing to the highest weight and LR-conditions, by the same logic used in \Cref{iC} and \Cref{iA}. Thus, we only have to show that both the slide to the left and the added entry preserve semistandardness in $Y\mapsto Y_0$ if and only if they do in $Y'\mapsto Y'_0$. \\

By applying the injective maps many times we delete every filling in $Y_0$ and $Y_0'$. Therefore, we can describe the tableaux by looking at sequences $Y_0 \mapsto Y=Y_1 \mapsto\cdots \mapsto Y_{N}$ and $Y'_0 \mapsto Y=Y'_1 \mapsto\cdots \mapsto Y'_{N}$.
Our crucial observation is that the maps in both side are recording the change to $\mu$. In the $sp$ side this is very direct, as $Y_k'\mapsto Y_{k+1}'$ deletes the filling $\overline{\,i_k}$ where $i_k$ is the row changing in $\mu$. 
Meanwhile, in the $LR$ side we have that $Y_k\mapsto Y_{k+1}$ deletes $m_k\in s_k$ while the entry $1\in s_k$ is in the $i_k$-th row.\\

To apply our cascading results, it is important to note that $\iota_{LR}^k$ is equivalent to first cascading the sequences $s_1, \cdots, s_k$ then deleting the last $k$ entries. Moreover, we can see inductively on $k$ that these sequences will still satisfy $(I)-(III)$ before deletion because, by semistandardness, the deleted fillings that appear before $m_k\in s_k$ are all bigger than $m_k$, while the ones that appear after $m_k\in s_k$ are all smaller than $m_k$.\\

We then have $\overline{\,i_1}\leq \overline{\,i_2}$ if and only if $i_1\ge i_2$ if and only if $1\in s_1$ is below $1\in s_2$. However, since $1\in s$ is always left-most, this is equivalent to $1\in  s_1$ appearing before $1\in s_2$ in the reading, so we can apply \Cref{cascades1}, to conclude this happen if and only if $m_1\leq m_2$.
Lastly, we set $r+1=\lambda_{\ell(\lambda)}$ and apply \Cref{cascades2}. For $r_0=1$, we obtain that $\overline{\,i_1}\leq \overline{\,i_{r+1}}$ if and only if $m_1\leq m_{r+1}$, which settles the equivalence of the semistandardness of the added filling. Moreover, for $r_0=r$, we get that $\overline{\,i_k}\leq \overline{\,i_{k+1}}$ for $k\leq r$ if and only if $m_k\leq m_{k+r}$ for $k\leq r$.
In other words, we conclude that the sliding in one side is semistandard precisely when it is in the other.
\end{proof}

\subsection{For \texorpdfstring{$\;\ell(\lambda)>n\;$}{l(lambda)}}
In the case of $\ell(\lambda)>n$,
some of the tableaux in
\[Im(F)=\begin{Bmatrix}
    \text{$sp$-highest weight}\\
    \text{semistandard tableaux}\\ 
    \text{of shape $\lambda$}\\ 
    \text{and weight $\mu$}
\end{Bmatrix}\]
might have fillings that are not $\Acal_n$, so only some of the tableaux in $Im(F)$ are in $\Bscr_{gl_{2n}}(\lambda)$.
\begin{ex}
    Let $\lambda=(3,2,1,1)$ and $\mu=(2,1)$, where $F$ is given by
    \[\ctableau
\begin{tikzcd}[row sep = 0.5em, ampersand replacement=\&]
    \adjustbox{scale=0.8}{\begin{ytableau}
    \none & \none & 1 \\
    \none & 2 \\
    3\\
    4\\
    \end{ytableau}}
    \arrow[r]\&
    \adjustbox{scale=0.8}{\begin{ytableau}
    \none & \none & 1 \\
    \none & 2 \\
    \overline{\,2\,}\\
    \overline{\,1\,}\\
    \end{ytableau}}
    \& \text{and} \&
    \adjustbox{scale=0.8}{\begin{ytableau}
    \none & \none & 1 \\
    \none &  2 \\
    1\\
    2\\
    \end{ytableau}}
    \arrow[r]\&
    \adjustbox{scale=0.8}{\begin{ytableau}
    \none & \none &1 \\
    \none &  \overline{\,1\,} \\
    3 \\
    \overline{\,3\,}\\
    \end{ytableau}}
\end{tikzcd}\]
Thus, the branching number is $2$ for $n\ge 3$, but only the first tableau should be counted for $n=2$.
\end{ex}
It is easy to see in our inductive process when these problematic fillings arise.
\begin{lemma}\label{1filling}
    The image $F(Y)$ has fillings in $\Acal_n$ if and only if no $\iota_{LR}^k(Y)$ has a $1$ filling below the $n$-th row.
\end{lemma}
\begin{proof}
    As we noted on \Cref{Fgood}, the fillings of the negative entries of $F(Y)$ correspond to the row number of the $1$s in the sequences altered by $\iota_{LR}$. Thus, for them to all be in $\Acal_n$, we need the $1$s to be above the $n$-th row. Furthermore, if all the negative fillings are in $\Acal_n$ then so are the positive ones, because the only unpaired entries are in $\mu \in X^+$ of $sp_{2n}$.
\end{proof}

This condition can also be seen in the final tableaux $Y$.

\begin{defi}\label{def: n-symp}
    We say a semistandard tableau $Y$ is $n$-symplectic if $Y(n+i, j)\ge2i$ for all $i,j$.
\end{defi}
\begin{lemma}\label{nsymp}
    The image $F(Y)$ has fillings in $\Acal_n$ if and only if $Y$ is $n$-symplectic.
\end{lemma}
\begin{proof}
    {$(\Leftarrow)$} If $F(Y)$ has a filling outside $\Acal_n$ then, by \Cref{1filling}, some $\iota^k_{LR}Y(n+i,j)=1$.\\
    Then, $Y(n+i,j)=m$ implies this box was cascaded $m -1$ times. So, it is part of a sequence until $2m$, whose last entry must be below the $(n+i+m)$-th row. That gives us an entry $Y(n+i', j')=2m \leq 2(m+i)\leq 2i'$.
    
    ${(\Rightarrow)}$ If $F(Y)$ has fillings in $\Acal_n$ then so does $F(\iota_{LR}(Y))=\iota_{sp}(F(Y))$. Then, by induction on $\Delta=|\lambda|-|\mu|$, we have $\iota_{LR}(Y)(n+i, j)>2i$. 
    Therefore, we only need to consider the fillings in the altered sequence.
    However, all but the non-deleted entries in $s$ decrease through $\iota_{LR}$, and therefore they also respect the inequality above. Moreover, by \Cref{1filling}, the $1$ is before the $n$-th row and the last entry is some $Y(n+i, 1)=m$, where 
    \[2(i-1)<\iota_{LR}(Y)(n+i-1,1)\leq m-2\vspace{-0.8cm}\]

\end{proof}

We then use the known branching rule to prove the crystal model for the symplectic branching.

\branching*
\begin{proof}
    The RHS is composed precisely by the tableaux of shape $\lambda$ and weight $\mu$ in $F(Y)$ with fillings in $\Acal_n$. Thus, by using \hyperref[sundaram]{Sundaram's branching rule} in one side, with \Cref{Fgood} and \Cref{nsymp} on the other, we obtain
    \[LHS=\sum_{\delta}(sp_nc)^{\lambda}_{\mu, (2\delta)'}= (ImF\;\cap \; \Bscr_{gl_{2n}}(\lambda))= RHS\]
    where $(sp_nc)^{\lambda}_{\mu, \nu}$ is the number of $n$-symplectic Littlewood--Richardson tableaux.
\end{proof}

\bibliographystyle{alpha}
\bibliography{ref}

\end{document}